\documentclass[reqno,centertags,12pt]{amsart}
\usepackage{amsmath,amsthm,amscd,amssymb,latexsym,verbatim}

\usepackage{graphicx,epsf,cite}

-\marginparwidth \oddsidemargin -4mm \evensidemargin \oddsidemargin

\newtheorem{theorem}{Theorem}[section]

\theoremstyle{definition}
\newtheorem{definition}[theorem]{Definition}

\theoremstyle{remark}

\newcounter{smalllist}

\DeclareMathOperator*{\sgn}{sgn}

\allowdisplaybreaks
\numberwithin{equation}{section}

\newcommand{\lb}{\label}

\newcommand{\beq}{\begin{equation}}
\newcommand{\eeq}{\end{equation}}

\newcommand{\bal}{\begin{align}}
\newcommand{\eal}{\end{align}}
\newcommand{\bals}{\begin{align*}}
\newcommand{\eals}{\end{align*}}

\newcommand{\bbR}{{\mathbb{R}}}

\newcommand{\bbZ}{{\mathbb{Z}}}

\newcommand{\bbS}{{\mathbb{S}}}

\newcommand{\calL}{{\mathcal L}}
\newcommand{\calK}{{\mathcal K}}

\begin{document}
\title[Virtual Linearity for KPP Reactions]
{Virtual Linearity for \\ KPP Reaction-Diffusion Equations}

\author{Andrej Zlato\v s}
\address{\noindent Department of Mathematics \\ UC San Diego \\ La Jolla, CA 92093, USA \newline Email:
zlatos@ucsd.edu}


\begin{abstract}
We show that long time solution dynamic for general reaction-advection-diffusion equations with KPP reactions is virtually linear in the following sense.  Its leading order depends on the non-linear reaction only through its linearization at $u=0$, and it can also be recovered for general initial data by  instead solving the PDE for restrictions of the initial condition to unit cubes on $\bbR^d$ (the latter means that non-linear interaction of these restricted solutions has only lower order effects on the overall solution dynamic).  The result holds under a uniform bound on the advection coefficient, which we show to be sharp.  
We also extend it to models with non-local diffusion and KPP reactions. 
\end{abstract}

\maketitle

\section{Introduction and Main Results} \lb{S1}

Many processes in nature are modeled by the reaction-diffusion equation
\beq\lb{1.1}
u_t = \calL u + f(t,x,u).
\eeq
The unknown function $u$ represents concentration of a substance or density of a species, which is subject to diffusion as well as some possibly space-time dependent reactive process (which may be a combination of birth and death processes), modeled by the two terms on the right-hand side of \eqref{1.1}.  
The basic case is $\calL=\Delta$, but 
when diffusion may be inhomogeneous, non-isotropic, and time-dependent, and an underlying advective motion may also be present (such as for processes occurring in fluid media), one instead considers the more general case
\beq\lb{1.2}
\calL u(t,x):= \sum_{i,j=1}^d A_{ij}(t,x) u_{x_ix_j}(t,x)  + \sum_{i=1}^d b_i(t,x) u_{x_i} (t,x).
\eeq
We will study here this model and its non-local version
\beq\lb{1.2a}
\calL u(t,x) := {\rm p.v.} \int_{\bbR^d} K (t,x,\nu) \left[ u(t,x+\nu)-u(t,x) \right] d\nu,
\eeq
with KPP (a.k.a.~Fisher-KPP)  reactions $f$.  Named after Kolmogorov, Petrovskii, and Piskunov \cite{KPP} and Fisher \cite{Fisher}, who first studied them in 1937, these reactions are defined as follows.

\begin{definition} \lb{D.1.0}
A Lipschitz function $f:\bbR^+\times\bbR^d\times [0,1]\to\bbR$ is a {\it KPP reaction} if $f(\cdot,\cdot,0)\equiv 0\equiv f(\cdot,\cdot,1)$ and $f(t,x,u)\le f_u(t,x,0)u $ for all $(t,x,u)\in\bbR^+\times \bbR^d\times [0,1]$ (with $f_u(\cdot,\cdot,0)$ existing pointwise), plus the following uniform hypotheses are satisfied. 
We have $\inf_{(t,x)\in\bbR^+\times\bbR^d} f(t,x,u)>0$ for each $u\in(0,1)$, as well as $\inf_{(t,x)\in\bbR^+\times\bbR^d} f_u(t,x,0)>0$ and   
\[
\lim_{u\to 0} \sup_{(t,x)\in\bbR^+\times\bbR^d} \left( f_u(t,x,0)-\frac{f(t,x,u)}u \right)=0.
\]
\end{definition}

Hence we will have $(t,x)\in \bbR^+\times\bbR^d$, and we define $f(t,x,\cdot)$ only on $[0,1]$ because we will consider here solutions $0\le u\le 1$ (i.e., with normalized concentration or density $u$).  While both $u\equiv 0$ and $u\equiv 1$ are stationary solutions for \eqref{1.1}, the former is unstable while the latter is asymptotically stable because $f>0$ on $\bbR^+\times\bbR^d\times(0,1)$, and one is interested in the transition of solutions from values near 0 to those near 1, which models invasion of the low concentration region $u\approx 0$ by the modeled substance or population.

\subsection{Virtual Linearity in the Classical Diffusion Case}

There is a vast literature on \eqref{1.1} with KPP reactions, and it would be futile to try to include here a list of even the most relevant works.  The reader can consult the reviews \cite{Berrev, Xinrev} and references therein, although many more important papers appeared since their publication.  When reviewing any of them, one notices that essentially all the results concerning asymptotic speeds of propagation of solutions (as opposed to, e.g., those relating to precise locations of the transition regions between values $u\approx 0$ and $u\approx 1$, such as in \cite{Bra, Lau, HNRR}) have one thing in common:  {\it they depend on $f$ only through $f_u(\cdot,\cdot,0)$.}  This is because the main KPP hypothesis $f(t,x,u)\le f_u(t,x,0)u $ means that the reaction strength $\frac {f(t,x,u)}u$   (i.e., the zeroth order coefficient when $u=u(t,x)$ is any solution to \eqref{1.1} and the PDE is viewed as a linear PDE for it) is greatest at values $u\approx 0$, so the leading order of the solution dynamic is determined at these values (and therefore spatially  at the leading edge of the invading population).  This is also sometimes referred to as {\it pulled} dynamic, as opposed ot the {\it pushed} dynamic for some other types of reactions, such as ignition or bistable, whose reaction strength is largest at intermediate values of $u$ (and therefore the solution dynamic is also primarily determined at these values \cite{ZlaInhomog, ZlaBist}).

Since $f(t,x,u)\approx f_u(t,x,0)u$ at $u\approx 0$, one might then think that the leading order of the solution dynamic for \eqref{1.1} with a KPP reaction is the same as for the linear PDE with $f_u(t,x,0)u$ in place of $f(t,x,u)$ (at least when considering the minimum of a solution of the latter and 1).  This is  not true in general because even if a spatial region has  large $f_u(\cdot,\cdot,0)$ for all time (such regions can drive the linear dynamic in all space-time), its effect on the non-linear dynamic  vanes once the solution reaches values away from 0 on that region.  Nevertheless,  dependence of the leading order of the solution dynamic on $f_u(\cdot,\cdot,0)$ only (rather than on all of $f$) is a version of linearity, and the author is not aware of any prior result that formally establishes this phenomenon for general KPP dynamics.   

Our first main result, Theorem \ref{T.1.1} below, therefore appears to be the first such result.  Moreover, it also demonstrates that \eqref{1.1} with KPP reactions shares another property with linear equations.  Namely, that the leading order of the solution dynamic for a general initial condition $u(0,\cdot)$ can be recovered from solving the PDE with initial conditions that are obtained by restricting $u(0,\cdot)$ to members of a partition of $\bbR^d$ into compact sets (we use below the unit cubes $C_n:= (n_1,n_1+1)\times\dots\times(n_d,n_d+1)$ with $n\in \bbZ^d$, but our proofs can be easily adapted to other choices).  That is, nonlinear interaction between the resulting initially compactly supported solutions does not affect the leading order of the solution dynamic, so in this sense the solution operator for \eqref{1.1} is also close to being linear.

This {\it virtual linearity} of \eqref{1.1} means that to investigate the leading order of the solution dynamic, it suffices to replace a general KPP reaction by some ``template'' reaction sharing the same $f_u(\cdot,\cdot,0)$ (e.g., $f'$ in Theorem \ref{T.1.1}), as well as to only consider solutions with initial data supported inside small compact sets.  We demonstrate it for general KPP reactions $f$ and uniformly elliptic diffusion matrices $A$, together with advection vectors $b$ that are uniformly bounded above by twice the square root of the product of the ellipticity constant of $A$ and $\inf f_u(\cdot,\cdot,0)$.  Moreover, this bound on $b$  is in fact sharp (see below).  

\begin{theorem} \lb{T.1.1}
Let  $f$ be a KPP reaction 
and $\calL$ be from \eqref{1.2}, where 
 $A=(A_{ij})$ is a bounded symmetric matrix with $A\ge \lambda I$ for some $\lambda>0$, and the vector $b=(b_1,\dots,b_d)$ satisfies $\|b\|_{L^\infty}^2< 4\lambda \inf_{(t,x)\in\bbR^+\times\bbR^d} f_u(t,x,0)$.
Let $g(u):=\min\{u,1-u\}$ and
 \beq \lb{2.30}
 f'(t,x,u):=f_u(t,x,0)g(u)
 \eeq
  for $(t,x,u)\in \bbR^+\times\bbR^d\times[0,1]$.
Then there is $\phi :\bbR^+\to\bbR^+$ with $\lim_{s\to \infty} \phi (s)=0$, 
and for each  $\delta\in(0,\frac 12]$ there is $\tau_\delta\ge 1$, 
such that  the following holds (see also Remark 3 below for the dependence of $\phi,\tau_\delta$ on $A,b,f$).  

If  $u:\bbR^+\times \bbR^d\to [0,1]$  solves \eqref{1.1}, and for each $n\in\bbZ^{d}$ we let $u_n' :\bbR^+\times \bbR^d\to[0,1]$ solve \eqref{1.1} with $f'$ in place of $f$ and with $u_n' (0,\cdot):=  u(0,\cdot) \chi_{C_n}$, 
then for each $(t,x)\in[\tau_\delta,\infty)\times \bbR^d$,
\beq  \lb{2.11}
\sup_{n\in\bbZ^d} u_n' \left( t-\delta t, x \right) -\phi(\delta t) \le u(t,x)  \le  \sup_{n\in\bbZ^d}  u_n'  \left(t+t^\delta,x \right)  +\phi(t^\delta).
\eeq
Moreover, if there is $\gamma>0$ such that for each $(t,x)\in \bbR^+\times \bbR^d$ the function $u\mapsto \frac {f(t,x,u)}u$ is non-increasing on $(0,\gamma ]$ and  $\sup_{u\in[\gamma ,\infty)} \frac {f(t,x,u)}u = \frac {f(t,x,\gamma )} {\gamma }$, then \eqref{2.11} also holds with $u_n'$ instead solving \eqref{1.1} with $f$, and with the first inequality being just $\sup_{n\in\bbZ^d} u_n' (t,x)\le u(t,x)$.
\end{theorem}

{\it Remarks.} 
1. So up to $o(t)$ time shifts (or even $o(t^{0+})$ time shifts in the second claim) and $o(1)$ errors, we have  $u\approx \sup_{n\in\bbZ^d} u_n' $. This of course also means that if $v$ is a solution to \eqref{1.1} with $f$ replaced by another KPP reaction that has the same linearization at $u=0$, then $v\approx u$ in the same sense provided $v$ has the same initial datum (and perturbations of the latter can also be handled easily, via the maximum principle).
\smallskip

2.  The proof can be easily adjusted so that both $t^\delta$ in \eqref{2.11} are replaced by $\delta t$.
\smallskip

3.  The proof shows that if  some non-decreasing function $\psi$ with   $\lim_{u\to 0}\psi(u)=0$ satisfies
\[
\psi(u)\ge  \sup_{(t,x)\in\bbR^+\times\bbR^d} \left(f_u(t,x,0)-\frac{f(t,x,u)}u \right),
 \]
some Lipschitz $f_0$ with $f_0'(0)> \frac 1{4\lambda}\|b\|_{L^\infty}^2$
satisifes $\inf_{(t,x)\in\bbR^+\times\bbR^d} f(t,x,u)\ge f_0(u)>0$  for each $u\in(0,1)$  (such $f_0$ exists because $\psi$ does), and we let  $0<B\le 2\sqrt{f_0'(0)\lambda} - \|b\|_{L^\infty} $ and 
\[
0<\alpha \le  
\min\left\{  \min_{i,j}\|A_{i,j}\|_{L^\infty}^{-1}, 
\, \|f_u(\cdot,\cdot,0)\|_{L^\infty}^{-1}
\right\},
\]
then in the first claim of the theorem, $\phi$ and $\tau_\delta$ depend only on $\alpha,\lambda,\psi,f_0,B,d$ 
(and $\tau_\delta$ also on $\delta$), while in the second claim they also depend on $\gamma$.\smallskip

4.  One could in fact replace $\sup_{n\in\bbZ^d} u_n'$ by  $\min \left\{ \sum_{n\in\bbZ^d} u_n',1 \right\}$ in \eqref{2.11}, which has a more ``linear'' feel but is also less convenient to use --- including in  applications of Theorem~\ref{T.1.1} to homogenization for KPP reaction-diffusion dynamics in random environments, which we provide in the companion papers \cite{ZhaZla5, ZlaKPPhomog}.
\smallskip

5.  Here $g$ could be any other $(t,x)$-independent KPP reaction with $g(u)\equiv u$ on $[0,\frac 12]$.  However,  we cannot use the linear dynamics, with $g(u)$ replaced by $u$ for all $u\ge 0$.  Indeed, consider for instance $\calL:=\Delta$ and $f_u(t,x,0)=1+C \chi_{B_1(0)}(x)$.  Then  the asymptotic speed of propagation of solutions to \eqref{1.1} with both $f$ and $f'$ is well known to be $2$, but this speed increases when $f$ is replaced by $f_u(t,x,0)u$ and $C\ge 0$ is large enough.  It is in fact 2 when $\lambda\le 2$ and  $\frac \lambda {\sqrt{\lambda-1}}$ otherwise, where $\lambda$ ($\to\infty$ as $C\to\infty$) is the principal eigenvalue of  $\Delta+1+C \chi_{B_1(0)}$ on $\bbR^d$ (with $\sqrt{\lambda-1}$ being the asymptotic exponential decay rate of the corresponding radial eigenfunction).
\smallskip

6. One  cannot hope for this result to extend to general non-KPP reactions.  This is the case even if we let $f':=f$; in this case $\sup_{n\in\bbZ^d} u_n' \le u$ is obvious but the second inequality in \eqref{2.11} need not hold even with $t^\delta$ replaced by $\delta t$ (see the remark after \cite[Theorem 1.4]{ZlaKPPhomog}).
\smallskip

The bound on $\|b\|_{L^\infty}$ in Theorem \ref{T.1.1} is sharp.  Indeed, consider $u_t=u_{xx}+\bar bu_x+g(u)$ with $\bar b>0$ a constant and $u(0,\cdot):=\chi_{(0,1)}$, in which case $2\sqrt{f_0'(0)\lambda}=2$ (taking $f_0:= g$).  It is well known (see, e.g., \cite{Lau}) that if $w$ solves this Cauchy problem without the first-order term, there is a strictly decreasing continuous function $q:(0,1)\to\bbR$ such that if we denote by  $x_\theta(t)>0$ the unique point with $w(t, x_\theta(t))=\theta=w(t, 1-x_\theta(t))$, then  for each $\theta_0\in(0,1)$ we have
\[
\lim_{t\to\infty} \sup_{\theta\in[\theta_0,1-\theta_0]} \left| x_\theta(t)- \left(2t-\frac 32\ln t +q(\theta) \right) \right|=0.
\]
Of course, the original PDE is solved by $u(t,x):=w(t,x+\bar bt)$; this also equals $u_0'$ in the theorem, while all other $u_n'$ are zero. So for $y_t:=(2-b)t-\frac 32\ln t +q(\frac 12)$ we have $\lim_{t\to\infty} u(t,y_t)=\frac 12$, while for each $\delta\in(0,1)$ we have 
\[
y_t=y_{t-\delta t} - (\bar b-2)\delta t - \frac 32 |\ln(1-\delta)| = y_{t+t^\delta} + (\bar b-2)t^\delta +\frac 32 \ln(1+t^{\delta-1}).
\]
Hence whenever $\bar b\in(2,2+\frac 4\delta)$, we obtain
\[
\lim_{t\to\infty} \sup_{n\in\bbZ} u_n'(t-\delta t,y_t) = 1  \qquad\text{and}\qquad 
\lim_{t\to\infty} \sup_{n\in\bbZ} u_n'(t+t^\delta,y_t)= 0,
\]
which contradicts both inequalities in \eqref{2.11} (and for any other KPP reaction $f$ with $f_u(0)=1$ we get the same result because the above asymptotics still hold, albeit with a different $q$).  Even for $\bar b=2$ we find that
\[
\lim_{t\to\infty} \sup_{n\in\bbZ} u_n'(t-\delta t,y_t)= q^{-1} \left( q \left( \frac 12 \right) - \frac 32 |\ln(1-\delta)| \right) > \frac 12,
\]
which still contradicts the first inequality (and if one wants the last claim in the theorem to technically not apply, it suffices to change the reaction on a short time interval $[0,t_0]$ and at all large $x$ so that the hypothesis of that claim is not satisfied but the change does not affect any point $(t,x,u(t,x))$; then the particular solution $u$ considered here is also unchanged).

Nevertheless, it is possible that a version of Theorem \ref{T.1.1} does hold with a larger uniform upper bound on $b$, provided $u_n'(t-\delta t,\cdot)$ and $u_n'(t+t^\delta,\cdot)$ are evaluated at some $(t,x,\delta)$-dependent points instead of at $x$.  Theorem \ref{T.2.1} below with $U\equiv 1\equiv U'$ would yield such a result if its hypotheses \eqref{2.3} and \eqref{2.3a} can be verified in some relevant setting.  A trivial example of this is the setting of Theorem \ref{T.1.1} with a large constant vector $\bar b$ added to $b$, when we clearly obtain \eqref{2.11} with $u_n'(t-\delta t,x+\bar b \delta t)$ and $u_n'(t+t^\delta,x-\bar bt^\delta)$ on the left and right.

\subsection{Extension to Non-local Diffusions}

Theorem \ref{T.1.1} extends to \eqref{1.1} with non-local diffusion operators from \eqref{1.2a} under suitable hypotheses.  Firstly, it is crucial that solutions do not propagate faster than ballistically, which requires the diffusion kernels $K$ to decay exponentially as $\nu\to\infty$.  Secondly, even when $K$ is close in some sense to being even in $\nu$ (here we will only assume it to be even)  one can only allow $O(|\nu|^{-d-2+\alpha})$ growth as $\nu\to 0$ (with some $\alpha>0$) if $\calL$ is to be well-defined.  This is the same growth as for $\calL=-(-\Delta)^{\alpha/2}$, for which well-posedness, comparison principle, and the parabolic Harnack inequality are known to hold \cite{BSV}.  The methods used to establish these should equally apply to various  kernels with $O(|\nu|^{-d-2+\alpha})$ asymptotics as $\nu\to 0$ that decay exponentially as $\nu\to\infty$.  However, instead of trying to prove them in any level of generality, we will state these properties as hypotheses  so that our result applies whenever these can be established.
We note that the Harnack inequality referred to here is the forward one (unlike for $\calL$ from \eqref{1.2}, equations with non-local diffusions may also satisfy backward-in-time Harnack inequalities, such as in \cite{BSV}).  Our main result in this setting is now the following analog of Theorem \ref{T.1.1}.


\begin{theorem} \lb{T.1.1a}
 Let $f$ be a KPP reaction and let $f'$ be from \eqref{2.30}.  Assume that $\calL$ is from \eqref{1.2a}, with $ K $  from some family $F$ of even-in-$\nu$ kernels such that for some $\alpha\in(0,1]$ and any $ K \in F$,  there is $\calK:(0,\infty)\to[0,\infty)$ with $\chi_{(0,\alpha]}(r)\le  \calK(r)\le \chi_{(0,\alpha]}(r)r^{-d-2+\alpha}$ on $(0,\infty)$ and 
\beq\lb{2.31}
\alpha  \calK(|\nu|) \le   K (t,x,\nu) \le \alpha^{-1} \max\left\{ \calK(|\nu|),  e^{-\alpha|\nu|} \right\} 
\eeq
for  each $(t,x,\nu)\in\bbR^+\times\bbR^{2d}$.  Assume that \eqref{1.1} with any such $ K $, any KPP reaction $f$, and locally BV initial data $0\le u(0,\cdot)\le 1$ is well-posed in some subspace $\mathcal A\subseteq L^{1}_{\rm loc}(\bbR^+\times \bbR^d)$, where the comparison principle for sub- and supersolutions to \eqref{1.1} as well as the parabolic (forward) Harnack inequality also hold (the latter with  uniform constants for all $K\in F$ and all $f$ with the same  Lipschitz constant), and the solutions for $u(0,\cdot)\equiv 0,1$ are $u\equiv 0,1$, respectively.   Then the claims in Theorem \ref{T.1.1}  hold for such $\calL$.
%
\end{theorem}

{\it Remark.}
One could also extend this result to mixed diffusion operators, with $\calL$ being the sum of the right-hand sides of \eqref{1.2} and  \eqref{1.2a}, but we will not do so here.
\smallskip


\subsection{Acknowledgements}
The author thanks Hongjie Dong, Jessica Lin, Jean-Michel Roquejoffre, and Luis Silvestre for useful discussions and pointers to literature.  He also acknowledges partial support by  NSF grant DMS-1900943 and by a Simons Fellowship.

\section{Classical Diffusion Case} \lb{S2}

The key to Theorem \ref{T.1.1} will be the following result, which provides a version of \eqref{2.11} on general $D\subseteq\bbR^d$ and without quantitative restrictions on $b$.  The price to pay is that the spatial arguments in the three expressions in \eqref{2.11} may be different, and one also needs to guarantee certain exponential growth of solutions from small initial data (which we later show to hold under the hypotheses of Theorem \ref{T.1.1}).  

Consider \eqref{1.1} on some open  $D\subseteq\bbR^d$, with $\calL$ from \eqref{1.2}.  
In the following theorem, all solutions are strong (from $W^{(1,2),d+1}_{\rm loc}(\bbR^+\times D)\cap C(\overline{\bbR^+\times D})$) and are assumed to 
satisfy homogeneous Dirichlet boundary conditions on $\bbR^+\times \partial D$.
We will call such solutions {\it SD solutions}. 
We note that the relevant well-posedness theory as well as comparison principle for sub- and supersolutions follow from, e.g., the corresponding linear theory in \cite[Chapter~7]{Lie} (specifically, Theorems 7.1 and 7.32).

\begin{theorem} \lb{T.2.1}
Let $\calL$ be given by \eqref{1.2}, with  $A=(A_{ij})$  a bounded symmetric
matrix
with $A\ge \lambda I$ for some $\lambda>0$, 
and  $b=(b_1,\dots,b_d)$  a bounded vector.
Let $f,f':\bbR^+\times D\times[0,\infty)\to\bbR$ be Lipschitz  with  $f(\cdot,\cdot,0)\equiv 0\equiv f'(\cdot,\cdot,0)$, and let $U,U':\bbR^+\times D\to[0,1]$ be some functions.
Assume that there are $\gamma \in(0,\frac 12]$
 and non-decreasing $\psi :(0,1)\to[0,\infty)$ with $\lim_{u\to 0} \psi (u)=0$
such that 
$0\le \frac{f'(\cdot,\cdot,u)}u - \frac{f(\cdot,\cdot,u)}u\le \psi (u)$ for each $u\in(0,\gamma ]$, 
for each $(t,x)\in \bbR^+\times D$ the function $u\mapsto \frac {f'(t,x,u)}u$ is non-increasing  on $(0,\gamma ]$ and  $\sup_{u\in[\gamma ,\infty)} \frac {\max\{f(t,x,u),f'(t,x,u)\}}u \le \frac {f'(t,x,\gamma )} {\gamma }$, and
\beq \lb{2.2}
\max\left\{  \max_{i,j}\|A_{i,j}\|_{L^\infty}, \,  \max_i\|b_i\|_{L^\infty}, \, \|f_u(\cdot,\cdot,0)\|_{L^\infty}
\right\} \le\gamma ^{-1}.
\eeq
Also assume that there are $\kappa>0$ and  $\phi :\bbR^+\to\bbR^+$ with $\lim_{s\to \infty} \phi (s)=0$,
and for each $t_0\ge 1$ and  $(s,x)\in\bbR^+\times D$ there are  $y_{t_0,x}^{s},y_{t_0+s,x}^{-s}\in D$, such that $y_{t_0+s,y_{t_0,x}^s}^{-s}=x$  and for any SD solutions $u,u':(t_0-1,\infty)\times D\to [0,1]$ to \eqref{1.1} and to \eqref{1.1} with $f'$ in place of $f$, respectively, we have
%
\begin{align}
u(t_0+s,x)\ge \min \left\{  e^{\kappa  s} u(t_0,y_{t_0+s,x}^{-s}), \, U(t_0+s,x)-\phi (s) \right\}, \lb{2.3}
\\ u'(t_0+s,x)\ge \min \left\{  e^{\kappa  s} u'(t_0,y_{t_0+s,x}^{-s}), \, U'(t_0+s,x)-\phi (s) \right\}. \lb{2.3a}
\end{align}
Let  $u:\bbR^+\times D\to [0,1]$ be an SD solution to \eqref{1.1},
and for each $n\in\bbZ^{d}$ let $u_n' :\bbR^+\times D\to[0,1]$ be the SD solution to \eqref{1.1} with $f'$ in place of $f$ and with $u_n' (0,\cdot):=  u(0,\cdot) \chi_{cC_n\cap D}$ for some $c>0$.  
Then for each  $\delta\in(0,\frac 12]$ 
there is $\tau_\delta\ge 1$ (depending also on $\gamma , \kappa,\psi ,c,d$) such that
\begin{align}
u(t,x) & \ge \min \left\{ \sup_{n\in\bbZ^d} u_n' \left(t-\delta t, y_{t,x}^{-\delta t} \right),\, U(t,x)  - \phi (\delta t) \right\},    \lb{2.4}
\\  \sup_{n\in\bbZ^d}  u_n'  \left(t+t^\delta,y_{t,x}^{t^\delta}\right)  & \ge \min \left\{ u(t,x) - t^{-1/\delta},\, U' \left(t+t^\delta,y_{t,x}^{t^\delta} \right) - \phi (t^\delta) \right\}   \lb{2.5}
\end{align}
for each $(t,x)\in[\tau_\delta,\infty)\times D$.
If $f'=f$, then clearly also 
\beq \lb{2.6}
u(t,x)\ge \sup_{n\in\bbZ^d}  u_n' (t,x).
\eeq
\end{theorem}

{\it Remarks.} 
1.  Note that the hypotheses on $f,f'$ guarantee existence of $f_u(t,x,0)=f_u'(t,x,0)$ as well as $\max\{f(t,x,u),f'(t,x,u)\}\le f_u(t,x,0)u$ for all $(t,x,u)\in\bbR^+\times D\times \bbR^+$. 
\smallskip

2.  A natural choice of $U,U'$ are some SD solutions to \eqref{1.1} and to \eqref{1.1} with $f'$ in place of $f$, respectively.  For instance,  when $D=\bbR^d$ and $f(\cdot,\cdot,1)\equiv 0 \equiv f(\cdot,\cdot,1)$, one might consider $U\equiv 1 \equiv U'$.
\smallskip

3. The restriction to $t_0\ge 1$ is needed so that parabolic Harnack inequality guarantees that $u(t_0,\cdot)$ does not vary too much near $y_{t_0+s,x}^{-s}$.  Otherwise hypothesis \eqref{2.3} could not hold for all $u$.\smallskip

4. The point $y_{t,x}^{-s}$ is such that the values of $u$ near $(y_{t,x}^{-s},t-s)$ provide a good lower bound for the value at $(t,x)$, and $y_{t,x}^{s}$ is such that  the values of $u$ near $(t,x)$ provide a good lower bound for the value at $(y_{t,x}^{s},t+s)$.  Of course, if $x\mapsto y_{t,x}^{-s}$ is a bijection on $D$ for some $t,s$, then $z\mapsto y_{t-s,z}^{s}$ must be its inverse.  One natural example of this is $y_{t,x}^s=x$, another is given by the ODE $\frac d{ds} y_{t,x}^s = b(t+s,y_{t,x}^s)$ with $y_{t,x}^0=x$.
\smallskip

5.  One might also consider other boundary conditions, provided the construction of exponentially-decaying ballistically-moving supersolutions from the proof below can be adjusted to that setting.

\begin{proof}
Comparison principle gives for each $s\in(0,\gamma ]$ and all $(t,x)\in\bbR^+\times D$ that
\beq\lb{2.7}
s e^{-\psi (s) t}\sup_{n\in\bbZ^d}  u_n' ( t,x)  \le u(t,x)\le \gamma ^{-1} \sum_{n\in\bbZ^d} u_n' (t,x).
\eeq
Indeed, the first inequality follows from $v(t,x):=s e^{-\psi (s) t} u_n' ( t,x) \,\,(\le s)$ being a subsolution to \eqref{1.1} for each $n\in\bbZ^d$, which holds because the hypotheses on $f,f'$ yield
\[
se^{-\psi (s) t} f'(t,x,u_n' ( t,x))\le f'(t,x, se^{-\psi (s) t} u_n' ( t,x)) \le f(t,x,se^{-\psi (s) t}u_n' ( t,x)) + \psi (s) se^{-\psi (s) t}u_n' ( t,x).
\]
  The second one follows from $v(t,x):= \gamma ^{-1} \sum_{n\in\bbZ^d} u_n' (t,x)$ being a supersolution to \eqref{1.1} in the region where $v\le 1$. 
This holds because if $0<v(t,x)\le 1$, then $u_n' (t,x)\le \gamma $ for each $n\in\bbZ^d$, so we can take $v_n:= \gamma ^{-1} u_n' (t,x)\le 1$ and apply the estimate
\[
 f \left(\cdot,\cdot, \sum_{n=1}^\infty v_n \right) 
=  \gamma ^{-1} \sum_{n=1}^\infty f \left(\cdot,\cdot, \sum_{m=1}^\infty v_m \right) \left( \sum_{m=1}^\infty v_m \right)^{-1} \gamma  v_n
 \le \gamma ^{-1} \sum_{n=1}^\infty f'(\cdot,\cdot, \gamma  v_n).
\]
This estimate holds for all $v_1,v_2,\dots\ge 0$ with $0<\sum_{n=1}^\infty v_n\le 1$, by the hypotheses on $f,f'$.

Next, without loss assume that $\kappa\le 1$.
Let us take any $\delta\in(0,\frac 12]$, pick $s\in(0, \gamma ]$ such that $\psi (s)\le \frac{\kappa  \delta} 2$, let $\tau_\delta:=\frac 2{\kappa  \delta} |\ln s|\ge 2$, and assume that  \eqref{2.4} fails for some $(t,x)\in[\tau_\delta,\infty)\times D$.  Then $u(t,x)<U(t,x)-\phi (\delta t)$, so \eqref{2.3} and \eqref{2.7}  yield
\[
u(t,x)\ge  e^{\kappa  \delta t} u \left( t-\delta t, y_{t,x}^{-\delta t} \right) \ge e^{\kappa  \delta t} s e^{- \kappa  \delta t/2}\sup_{n\in\bbZ^d}  u_n'  \left( t-\delta t, y_{t,x}^{-\delta t} \right) \ge \sup_{n\in\bbZ^d}  u_n'  \left( t-\delta t, y_{t,x}^{-\delta t} \right).
\]
But this means that  \eqref{2.4} does hold for $(t,x)$, a contradiction.  Hence \eqref{2.4} must hold.

To prove \eqref{2.5}, consider any $\delta\in(0,\frac 12]$, let $a:=\gamma ^{-1}(1+d+d^2)$ and  fix any $(t',x')\in\bbR^+\times D$.  For each $n\in\bbZ^d$ let $x_n$ be the point from $c\overline C_n$ closest to $x'$ and let $e_n:=\frac{x'-x_n}{|x'-x_n|}$ when $x_n\neq x'$ (otherwise pick any $e_n\in\bbS^{d-1}$).  Since $v_n(t,x):=e^{a t-(x-x_n)\cdot e_n}$ is a supersolution to \eqref{1.1} by \eqref{2.2}, and $v_n(0,\cdot)\ge \chi_{cC_n}\ge u_n' (0,\cdot)$, we have
\[
\gamma ^{-1}u_n' (t',x')\le \gamma ^{-1} v_n(t',x') \le \gamma ^{-1} e^{a t' -  |x'-x_n|}.
\]
Sum of the right-hand sides over all $n$ with $|x'-x_n|\ge (a + 1) t'$ is less than $\frac {1}2 (t')^{-1/\delta}$ whenever $t'\ge \tau$ for some $(\delta,\gamma ,c,d)$-dependent $\tau\ge 1$.  

So if  \eqref{2.5} fails for some $(t',x')\in[ \tau ,\infty)\times D$, then $u(t',x')\ge (t')^{-1/\delta}$ and \eqref{2.7} show that there is $n$  with $|x'-x_n|< (a + 1) t'$ such that
\beq\lb{2.10}
\gamma ^{-1} u_n' (t',x')\ge \frac {1} 2 (t')^{-1/\delta}  \left((a + 1) t'+ c \sqrt d \right)^{-d} c^{d} V_d^{-1} ,
\eeq
where $V_d$ is the volume of the unit ball in $\bbR^d$.
We also have 
\[
u_n'  \left(t'+(t')^\delta, y_{t',x'}^{(t')^\delta} \right) < U' \left(t'+(t')^\delta, y_{t',x'}^{(t')^\delta} \right) - \phi ((t')^\delta)
\]
 because \eqref{2.5} fails for $(t',x')$, so \eqref{2.3a}  yields
\[
u_n'  \left(t'+(t')^\delta, y_{t',x'}^{(t')^\delta} \right)  \ge e^{\kappa  (t')^\delta} u_n' (t',x')
\]
 (recall also that $y_{t+s,y_{t,x}^s}^{-s}=x$). 
But  \eqref{2.10} shows that the right-hand side is greater than $1$
whenever $t'\ge\tau_\delta$, with some $(\delta,\gamma ,\kappa,c,d)$-dependent $\tau_\delta\ge\tau$.  So \eqref{2.5} must hold for all  $(t',x')\in[ \tau_\delta ,\infty)\times D$, otherwise we would have a contradiction.

Finally, comparison principle immediately yields \eqref{2.6}.  
%
\end{proof}

We note that when $u\mapsto \frac {f(t,x,u)}u$ is non-increasing  on $(0,1]$, then one can easily show that with $f'=f$ we have $\sup_{n\in\bbZ^d} u_n'\le u\le\sum_{n\in\bbZ^d} u_n'$.  For $d=1$, $\calL=\partial_{xx}$, $f(t,x,u)=f(u)$, and a sum of two initial data (i.e., $u(0,\cdot)=u_1(0,\cdot)+u_2(0,\cdot)$), this has already appeared in \cite[Lemma 8.4]{Uch} and  \cite[Lemma 3.5]{Bra}.

We are now ready to prove Theorem \ref{T.1.1}.

\begin{proof}[Proof of Theorem \ref{T.1.1}]
Let $f_0, \psi,B,\alpha$ be from Remark 3 after Theorem \ref{T.1.1}.  We will denote $\beta:=f_0'(0)>0$ because it will be used extensively, and let 
\[
\gamma:=\min \left\{\alpha, \,\frac 1{2\sqrt{\beta\lambda}}  ,\, \frac 12 \right\}>0
\]
 (when proving the last claim, add $\gamma$ from it into the  $\min$).  Note that we assume 
 \[
 \max_i\|b_i\|_{L^\infty}\le 2\sqrt{\beta\lambda}\le\gamma^{-1}.
 \]  
We will  show that Theorem \ref{T.2.1} applies with $y_{t_0,x}^{s}:=x=:y_{t_0+s,x}^{-s}$ for all $(t_0,s,x)\in[1,\infty)\times\bbR^+\times\bbR^d$, with $U\equiv 1\equiv U'$, and some $\kappa,\phi$ (when proving the last claim, we use instead $f':=f$ in Theorem \ref{T.2.1}).  
Then \eqref{2.4} and \eqref{2.5} will immediately yield \eqref{2.11}, provided we replace $\phi(t)$ by $\phi(t)+t^{-1/2}$.

The only hypotheses of  Theorem \ref{T.2.1} that do not obviously hold are \eqref{2.3} and \eqref{2.3a}. The proofs are identical, so let us prove \eqref{2.3}.  
The key will be the following construction of a sequence of appropriate subsolutions $\{u_{v,j}\}_{j=0}^\infty$ to \eqref{1.1}, which are essentially ballistically spreading radially symmetric plateaus (with a common positive spreading speed) at levels converging to 1 as $j\to\infty$, with the first of them starting at level $v>0$ (which can be chosen to be arbitrarily small).  We will then show that any solution $0\le u\le 1$ to \eqref{1.1} that starts above the first subsolution will not only remain above it forever, but it will also be eventually pushed above any other of the subsolutions --- and \eqref{2.3} will follow. 
In the whole proof,  all constants and functions may depend on $\gamma,\beta,\lambda,\psi,B,f_0,d$, with only the constant $v$ below (and constants/functions with $v$ in their subscripts) possibly depending on the solution $u$. 

Let $\Lambda := \frac d\alpha$, so $A\le \Lambda I$.
If $u(t,x):=w(t,|x|)$ for $x\in\bbR^d$ and some $w:\bbR^+\times \bbR\to [0,\infty)$, then we have (with arguments $(t,x)$ and $(t,y):=(t,|x|)$, as appropriate)
\begin{align*}
\sum_{i,j=1}^d A_{ij} u_{x_ix_j} = \frac{w_{yy} |x|-w_y}{|x|^3} \sum_{i,j=1}^d A_{ij} x_ix_j + \frac {d}{|x|} w_y
\ge \min\{\lambda, \Lambda \sgn w_{yy}\} |w_{yy}| - \frac {d+\Lambda}{|x|} |w_y|.
\end{align*}
This and $\|b\|_\infty\le 2\sqrt{\beta\lambda}- B$ yield for  $C:=\frac B4 \sqrt{\beta/\lambda}$ ($\le \frac\beta 2$ by $B\le 2\sqrt{\beta\lambda}$) and all $|x|\ge \frac {3(d+\Lambda)}B$,
\beq\lb{2.12}
\calL u  + (\beta-C) u-\frac B3|\nabla u| \ge \min \left\{\lambda, \Lambda \sgn w_{yy}\right\} |w_{yy}| - \left(2\sqrt{\beta\lambda}- \frac B3  \right) |w_y| + \left( \beta - C \right) w .
\eeq

Let now $z$ be the root of $\lambda z^2+ (2\sqrt{\beta\lambda}- \frac B3) z + \beta - C =0$ that lies in the second quadrant of the complex plane; note that the discriminant of this quadratic equation is
\[
4\lambda C - \frac B9 \left(12\sqrt{\beta\lambda}-B \right) \le B \sqrt{\beta\lambda} - \frac B9 10\sqrt{\beta\lambda} =- \frac B9 \sqrt{\beta\lambda}<0.
\]
Then the function $\tilde\xi (y):=e^{y\,{\rm Re}\, z}\sin (y\, {\rm Im}\, z)$ solves $\lambda \tilde\xi ''+ (2\sqrt{\beta\lambda}- \frac B3) \tilde\xi ' + (\beta - C)\tilde\xi  =0$.  Let $ y_3:=\frac\pi{{\rm Im}\, z}$ and let $y_2\in(0, y_3)$ be its greatest inflection point that is smaller than $ y_3$, and let
\[
\xi(y):=
\begin{cases}
0 & y>  y_3, \\
\tilde\xi (y) & y\in[y_2, y_3], \\
\tilde\xi (y_2) + \tilde\xi '(y_2)(y-y_2) & y<y_2.
\end{cases}
\]
Since $\tilde\xi $ is convex and positive on $[y_2, y_3)$, so is $\xi$ on $(-\infty, y_3)$.  Hence $\xi''\ge 0\ge \xi'$ (in the sense of distributions due to the discontinuity of $\tilde\xi $ at $ y_3$; elsewhere it is $C^2$) and so
\beq\lb{2.13}
 \lambda |\xi''| - \left(2\sqrt{\beta\lambda}- \frac B3  \right) |\xi'| + \left( \beta - C \right) \xi \ge 0.
\eeq
Indeed, this clearly  holds on $[y_2,\infty)$, while on $(-\infty,y_2]$ it follows from 
\beq\lb{2.14}
\left(2\sqrt{\beta\lambda}- \frac B3\right) \xi' (y_2)+ (\beta - C)\xi(y_2) =0
\eeq
because on this interval $\xi''\equiv 0$, $\xi'\equiv\xi'(y_2)$, and $\xi$ is decreasing.  Then of course the left-hand side of \eqref{2.14} equals $(\beta-C)\xi'(y_2)(y-y_2)\ge 0$ on $(-\infty,y_2]$, and it is  easy to show that there are $y_0<y_1<y_2$ and $\zeta:\bbR\to[0,\infty)$ that coincides with $\xi$ on $[y_1,\infty)$, is concave and $C^2$ on $(-\infty,y_2]$, constant on $(-\infty,y_0]$ (and hence also maximal there), and satisfies
\beq\lb{2.15}
\min \left\{\lambda, \Lambda \sgn \zeta''\right\} |\zeta''| - \left(2\sqrt{\beta\lambda}- \frac B3  \right) |\zeta'| + \left( \beta - C \right) \zeta \ge 0.
 \eeq
 
 Let now $v_0>0$ be such that $\psi(v_0)\le \frac C 2$  and for any $v\in(0,v_0]$ let 
\[
 w_v(t,y):=\frac {\min\{e^{Ct/2}v, v_0 \}} {\zeta(y_0)}  \, \zeta \left( y+y_0-\frac {3(d+\Lambda)}B-\frac B3 t \right).
 \]
   Then $0\le w_v \le v_0$ because $0\le\zeta\le\zeta(y_0)$, and $w_v(t,\cdot)$ is constant on $[-\infty,\frac {3(d+\Lambda)}B]$ for each $t\ge 0$.
This, \eqref{2.12}, and \eqref{2.15} show that  $u_{v,0}(t,x):=w_v(t,|x|)$ satisfies on $\bbR^+\times\bbR^d$,
\[
\calL u_{v,0}  + f(t,x,u_{v,0}) - (u_{v,0})_t \ge  \calL u_{v,0}  + (\beta-C) u_{v,0}-\frac B3|\nabla u_{v,0}|  \ge 0.
\]
Hence $u_{v,0}$ is a subsolution to \eqref{1.1} such that at each $t\ge 0$ we have
 \beq\lb{2.16}
\min\{e^{Ct/2}v, v_0\} \chi_{B_{p+q t}(0)} \le u_{v,0}(t,\cdot) 
\le \min\{e^{Ct/2}v, v_0\} \chi_{B_{y_3-y_0+p+q t}(0)} ,
\eeq
where $p:= \frac{3(d+\Lambda)}B$ and $q:=\frac B3$.

Let now $t_{v,0}:=\max\{ \frac 2C \ln \frac{v_0}v, \frac {y_3-y_0} {q} \}$, so that 
 \beq\lb{2.17}
v_0 \chi_{B_{y_3-y_0+p+q (t-t_{v,0})}(0)} \le u_{v,0} (t,\cdot)  \le v_0
\eeq
for all $t\ge t_{v,0}$.  Let $L\ge 1$ be a Lipschitz constant for $f_0$ and let $r:=\frac{f_0(v_0)}{2\beta -C +L}$ ($<\frac{f_0(v_0)}L\le v_0$).  It follows that $f_0(v')\ge (\beta-\frac C2)(v'-(v_0-r))$ for all $v'\in[v_0-r,v_0+r]$, and so as in the above argument, we obtain that
 \[
 u_v'(t,x):=v_0-r+\frac {\min\{e^{Ct/2}r, 2r \}} {\zeta(y_0)}  \, \zeta \left( |x| +y_0-p- q (t-t_{v,0}) \right)
 \]
 is a subsolution to \eqref{1.1} on time interval $[t_{v,0},\infty)$.  Since $u_v'\le v_0\le  u_{v,0}$ on $\{t_{v,0}\}\times B_{y_3-y_0+p}(0)$ and $u_v'(t,\cdot) = v_0-r <v_0=  u_{v,0}(t,\cdot)$ on $ \partial B_{y_3-y_0+p+q(t-t_{v,0})}(0)$ for each $t\ge t_{v,0}$, it follows that
 \[
u_{v,1}:= \max \left\{u_{v,0}, u_v'\chi_{\{|x|<y_3-y_0+p+q(t-t_{v,0})\}} \right\}
 \]
 is a subsolution to \eqref{1.1} on  time interval $[t_{v,0},\infty)$, with $u_{v,1}(t_{v,0},\cdot)\le u_{v,0}(t_{v,0},\cdot)$.  Also, if we let $v_1:=v_0+r$ and $t_{v,1}:=t_{v,0}+\sigma$, where $\sigma:=\max\{ \frac 2C \ln 2, \frac {y_3-y_0} {q} \}$, then for $t\ge t_{v,1}$ we have
\[
v_1 \chi_{B_{y_3-y_0+p+q (t-t_{v,1})}(0)} \le u_{v,1}(t,\cdot)  \le v_1.
\]

We can use this in place of  \eqref{2.17} to similarly obtain a subsolution $u_{v,2}$ to \eqref{1.1} on  time interval $[t_{v,1},\infty)$, such that $u_{v,2}(t_{v,1},\cdot)\le u_{v,1}(t_{v,1},\cdot)$ and for all $t\ge t_{v,2}$ we have
\[
v_2 \chi_{B_{y_3-y_0+p+q (t-t_{v,2})}(0)} \le u_{v,2}(t,\cdot)  \le v_2,
\]
where $v_2:=v_1+\frac{f_0(v_1)}{2\beta -C +L}$ and $t_{v,2}:=t_{v,1}+\sigma$. Repeating this argument, we obtain a $v$-independent  sequence $0<v_1<v_2<\dots$ converging to 1 (because Lipschitz $f_0>0$ on $(0,1)$) and subsolutions $\{u_{v,k}\}_{k\ge 0}$ to \eqref{1.1} on intervals $[t_{v,k-1},\infty)$, with $t_{v,k}:=t_{v,0}+ k\sigma$ and $t_{v,-1}:=0$, such that for each $k\ge 1$ we have $u_{v,k} (t_{v,k-1},\cdot)\le u_{v,k-1}(t_{v,k-1},\cdot)$ and
\beq\lb{2.18}
v_k \chi_{B_{y_3-y_0+p+q (t-t_{v,k})}(0)} \le u_{v,k} (t,\cdot)  \le v_k
\eeq
for all $t\ge t_{v,k}$.  These subsolutions will now allow us to obtain \eqref{2.3}.

Let  $u:(t_{0}-1,\infty)\times\bbR^d\to [0,1]$ solve \eqref{1.1} and pick any $x\in\bbR^d$.  Shift $A,b,f,u$ by $(-t_{0}-1,-x)$ in space-time so that we have $u:(-2,\infty)\times\bbR^d\to [0,1]$, and we therefore need to prove \eqref{2.3} with $(t_{0},x)=(-1,0)$ (note that the above subsolutions are also subsolutions for all space-time translations of $A,b,f$).  This is then the estimate
\beq\lb{2.19}
u(s-1,0)\ge \min \left\{  e^{\kappa  s} u(-1,0), \, 1-\phi (s) \right\}
\eeq
for some $u$-independent $\kappa,\phi$ as required  and all $s\ge 0$.
Assume also that $u\not\equiv 0$ because otherwise this  holds trivially.  

By the parabolic Harnack inequality \cite[Corollary 7.42]{Lie}, $u\ge 0$, and $\big|\frac {f(\cdot,\cdot,u)}u \big|\le\frac 1\gamma$ for all $u\in(0,1]$ (since $f$ is KPP), there is $u$-independent  $\mu\in(0,v_0e^{-C(y_3-y_0)/2q}]$ such that $u(0,\cdot)\ge \mu u(-1,0) \chi_{B_{y_3-y_0+p}(0)}$ (note that dependence on $\lambda$ in Remark 2 after Theorem \ref{T.1.1} enters through this $\mu$).
If we let $v:=\mu u(-1,0)\le v_0$, then  $u(0,\cdot)\ge u_{v,0}(0,\cdot)$, so $u\ge u_{v,0}$ on $[0,\infty)\times\bbR^d$.  
Thus $u(t_{v,0},\cdot)\ge u_{v,0}(t_{v,0},\cdot)\ge u_{v,1}(t_{v,0},\cdot)$, so $u\ge u_{v,1}$ on $[t_{v,0},\infty)\times\bbR^d$.  Continuing this way, we find that $u\ge u_{v,k}$ on $[t_{v,k-1},\infty)\times\bbR^d$ for each $k$, hence
\beq\lb{2.20a}
u(t,0)\ge \max \left \{ \min\{e^{Ct/2} v, v_0\}, v_1 H(t-t_{v,1}), v_2 H(t-t_{v,2}),\dots \right \} 
\eeq
for all $t\ge 0$, where $H:=\chi_{[0,\infty)}$.  Now let $\kappa:=\frac C{16}$ and for each $s\ge 0$  define 
\[
\phi(s):=
\begin{cases} 1 & s< s_0:=\max \left\{2,\frac {16}C\ln \frac{1}\mu, 2(t_{v,1}-t_{v,0})+1 \right\}, \\
1-v_{\max\{ k \,|\, s\ge 2(t_{v,k}-t_{v,0})+1\}} &  s\ge s_0,
\end{cases}
\]
which converges to 0 as $s\to\infty$ because $v_k\to 1$.
Since $t_{v,k}-t_{v,0}$ ($=k\sigma$) is independent of $v$ (and hence of $u$) for each $k$, so is $\phi$.  Also note that $t_{v,0}= \frac 2C \ln \frac{v_0}v$ (and hence $e^{C  t_{v,0}/2} v =v_0$) because $\mu\le v_0e^{-C(y_3-y_0)/2q}$.

Then \eqref{2.19} obviously holds for all $s<s_0$, so let us assume that $s\ge s_0$   and let $k:=\max\{ j \,|\, s\ge 2(t_{v,j}-t_{v,0})+1\}$ (so $\phi(s)=1-v_{k})$.  If \eqref{2.19} fails, then \eqref{2.20a} implies that $s<t_{v,k}+1$, and hence $t_{v,k}<2t_{v,0}$ by the definition of $k$.  But then $s\in[s_0,2t_{v,0}+1]$, so
\[
e^{\kappa  s} u(-1,0) = e^{C  s/16} v\mu^{-1}  \le e^{C s/8} v \le e^{C (s-1)/4} v \le \min\{e^{C(s-1)/2} v, v_0\} \le u(s-1,0)
\]
by \eqref{2.20a}.
Hence \eqref{2.19} also holds for all $s\ge s_0$, and the proof is finished.
 \end{proof}

Note that the proof shows that $\phi$ decays exponentially as $s\to\infty$ if $\liminf_{v\to 1} \frac {f_0(v)}{1-v} >0$.

\section{Non-local Diffusion Case} \lb{S3}

Let us now consider \eqref{1.1} on  $\bbR^d$, with $\calL$ from \eqref{1.2a}.   Any solutions considered here will be from the space $\mathcal A$ from Theorem \ref{T.1.1a}.

\begin{theorem} \lb{T.2.2}
Theorem \ref{T.2.1} holds when $f,f'$ are as in that theorem, $\calL$ is from \eqref{1.2a} instead of \eqref{1.2}, hypotheses on $\calL$ and \eqref{2.2} are replaced by $ K $ being even in $\nu$ and 
\[
\max\left\{  \sup_{(t,x,\nu)\in\bbR^+\times \bbR^{2d}}  K (t,x,\nu) \max \left\{   |\nu|^{d+2-\gamma}, e^{\gamma|\nu|} \right\},  \, \sup_{(t,x)\in\bbR^+\times \bbR^d} f_u(t,x,0)
\right\} \le\gamma ^{-1},
\]
and well-posedness and comparison principle hold for \eqref{1.1} on the space $\mathcal A$ from Theorem~\ref{T.1.1a}.
\end{theorem}

\begin{proof}
The proof is identical to that of Theorem \ref{T.2.1}, with the only change being the choice of the exponential supersolutions $v_n$, which will now instead be $v_n(t,x):=e^{a t-\gamma(x-x_n)\cdot e_n/2}$ with some $\gamma$-dependent $a>0$.
This is because any  $u\in C^3(\bbR^d)$ satisfies
\[
u(x+\nu)+u(x-\nu)-2u(x)= 2 \nu\cdot D^2u(x)\nu + O(|\nu|^3),
\]
so $ K $ being even in $\nu$ and local integrability of $|\nu|^{-d+\gamma}$ imply that $(\calL+\gamma) v_n\le a v_n$ holds with some $a>0$ depending only on $\gamma$.
\end{proof}


\begin{proof}[Proof of Theorem \ref{T.1.1a}]
This proof tracks that of  Theorem  \ref{T.1.1}, first picking  $\beta,f_0,\psi,\gamma$ in the same way and then using Theorem \ref{T.2.2} instead of Theorem \ref{T.2.1}.  One then again only needs to prove \eqref{2.3} with $y_{t_0,x}^{s}:=x=:y_{t_0+s,x}^{-s}$,
and $U\equiv 1$, for some $\kappa,\phi$.  This is  done via a similar construction of an appropriate sequence of  subsolutions to \eqref{1.1}.

Let $\zeta:\bbR \to [-\frac 12, \frac 12]$ be such that $\|\zeta^{(j)}\|_{L^\infty}\le 1$ for $j=1,2,3$ and for some $0=y_0<y_1<y_2<y_3<y_4$ we have that $\zeta\equiv \frac 12$ on $(-\infty, y_0]$ and $\zeta\equiv -\frac 12$ on $[ y_4,\infty)$, also $\zeta$ is concave on $(-\infty, y_1]$ and convex on $[y_1,\infty)$, as well as $\zeta(y_2)=\frac 14$ and $\zeta(y_3)=0$ and $\zeta''\equiv \frac 1{100}$ on $[y_2,y_3]$. 
Taylor's theorem then shows that for any $\eta,p>0$, the function $u(x):=\zeta(\eta |x|-p)$ satisfies
\[
\left| u(x+\nu)+u(x-\nu)-2u(x) - 2 \nu\cdot D^2u(x)\nu \right| \le 2 c_d \eta^{3}|\nu|^3
\]
for some $c_d>0$ only depending on $d$, and for all $x,\nu\in\bbR^d$.  A simple computation shows that for any $x,\nu\in\bbR^d$ we also have
\[
 \nu\cdot D^2u(x)\nu= \left( \eta^2\zeta''(\eta |x|-p) - \frac{\eta\zeta'(\eta |x|-p)}{|x|} \right) \left( \frac x{|x|}\cdot \nu \right)^2.
 \]
If we take $p\ge 200$, then the first parenthesis above is zero when $|x|\le \frac {200}\eta$, while for $|x|\ge \frac {200}\eta$ it is bounded  below by $- \frac{3\eta^2}{200}$, and for $|x|\in[\frac{y_2+p}\eta,\frac{y_3+p}\eta]$ it is bounded below by $\frac{\eta^2}{200}$.
This, evenness of $K$ in $\nu$, and \eqref{2.31} (recall also that we have chosen $\gamma$ to be $\le\alpha$) show that at any $t> 0$ we have for any $R\ge 0$,
\beq\lb{4.1}
\calL u(x)  \ge -   
\int_{B_R(0)}  \frac {\max\{|\nu|^{-d-2+\gamma}, 1\}}\gamma \left( \frac{3\eta^2 |\nu|^2}{200} + c_d \eta^3 |\nu|^3 \right)d\nu 
- \int_{\bbR^d \setminus B_R(0)} \frac {e^{-\gamma|\nu|}}\gamma d\nu
\eeq
for all $x\in\bbR^d$ (recall also that $u(x+\nu)-u(x)\ge -1$), but also
\beq\lb{4.2}
\calL u(x)  \ge   
\int_{B_\gamma(0)}   \calK(|\nu|) \left[ \gamma\frac{\eta^2}{200}  \left( \frac x{|x|}\cdot \nu \right)^2 - \frac{ c_d \eta^3 |\nu|^3} \gamma \right] d\nu  
- \int_{\bbR^d}  \frac   {e^{-\gamma|\nu|} }\gamma  \min\{ c_d \eta^3 |\nu|^3,1\}  d\nu 
\eeq
when $|x|\in[\frac{y_2+p}\eta,\frac{y_3+p}\eta]$ (note that both these lower bounds are independent of $x$).

Let now $C':=\frac 12\inf_{u\in(0,1/2]} \frac {f_0(u)}u>0$, pick $R$ so that the second integral in \eqref{4.1} is $\le \frac {C'}{8}$, and then $\eta\in(0,1]$ so that the first integral is $\le \frac {C'}{8}$.  Then \eqref{4.1} shows that 
\beq\lb{4.3}
\calL u(x) +2C'u(x) \ge  C' u(x)
\eeq
whenever $u(x)\ge \frac 14$.  Next decrease $\eta$ further (this will not compromise \eqref{4.3}) so that the right-hand side of \eqref{4.2} becomes some $C''>0$ (this is possible because $\calK(r)\ge 1$ for $r\le\gamma$).
Then \eqref{4.2} shows that
\beq\lb{4.4}
\calL u(x) \ge C''
\eeq
whenever $u(x)\in[0, \frac 14]$.  But this and $\eta\le\frac 14$ mean that if we let $C:=\frac 12 \min\{C',C''\}>0$, then
\beq\lb{4.5}
\calL u(x) + 2C'u(x)\ge C \left( |\nabla u(x)|+u(x) \right)
\eeq
holds whenever $u(x)\ge 0$.  This and the definition of $C'$ show that if $v_0:=\frac 12$, $v\in(0,\frac 12]$, and
\[
u_{v,0}(t,x):=  \left(  \min\{e^{Ct/2}v, v_0 \} \, 2 \zeta(\eta |x|-200- Ct)  \right)_+,
\]
then $u_{v,0}$ is a subsolution to \eqref{1.1} on $\bbR^+\times\bbR^d$.  The factor $\frac 12$ in the exponent is not necessary but we added it to make $u_{v,0}$ similar to $u_{v,0}$ in the proof of Theorem \ref{T.1.1}.

In particular, we now have 
 \[
 v_0 \chi_{B_{(200+C t)/\eta }(0)} \le u_{v,0} (t,\cdot)  \le v_0
 \]
 for all $t\ge \frac 2C \ln \frac {v_0}v$.
As in the proof of Theorem \ref{T.1.1}, we can now find $r,t_{v,0}> 0$ (then we let $v_1:=v_0+r$) and a subsolution 
 \[
 u_v'(t,x):=v_0-r+  \left( \min\{e^{Ct/2}r, 2r \}  \, 2 \zeta \left( \eta |x| -200- C (t-t_{v,0}) \right)  \right)_+
 \]
to \eqref{1.1}  on time interval $[t_{v,0},\infty)$ such that
 \[
u_{v,1}:= \max \left\{u_{v,0}, u_v'\chi_{\{ |x|<(200+Ct)/\eta\}} \right\}
 \]
 is also a subsolution to \eqref{1.1} on time interval $[t_{v,0},\infty)$.  We note that non-locality of $\calL$ causes a minor issue here but this can be easily overcome by halving $C$ above  so that we can add $Cu(x)$ to the right-hand side of  \eqref{4.5}.  This gives us an extra term $C (u_v'(t,x)-(v_0-r))$ in the same estimate for $u_v'$, which is no less than $Cr$ at all points where $u_{v,1}>u_{v,0}$ (all these have $|x|<(y_2+200+ C (t-t_{v,0}))/\eta$).  This will dominate the decrease of $\calL u_{v'}$ at these points caused by replacing $u_{v'}$ by $u_{v'} \chi_{\{(200+Ct)/\eta\}}$, provided $t_{v,0}$ is chosen large enough (recall that $K$ has a uniform  exponential decay as $\nu\to\infty$).  
 
 The construction of subsolutions $u_{v,2}, u_{v,3},\dots$ is analogous  (without the need to change $C$ further), with the corresponding times $t_{v,1},t_{v,2},\dots$
 not anymore forming an arithmetic sequence, but with $t_{v,k}-t_{v,0}$ again independent of $v$.  We again have $u_{v,k}(t_{v,k-1},\cdot)\le u_{v,k-1}(t_{v,k-1},\cdot)$ for all $k\ge 1$, and \eqref{2.18} instead becomes
 \[
 v_k \chi_{B_{(200+C (t-t_{v,k-1}))/\eta}(0)} \le u_{v,k} (t,\cdot)  \le v_k
 \]
  for all $t\ge t_{v,k-1}+\frac 2C \ln 2$.  The rest of the proof is identical to that of Theorem \ref{T.1.1} (in particular, the Harnack inequality for $\calL$ is used in this part).
\end{proof}

\end{document}